\documentclass[a4paper,english,fontsize=10pt,parskip=half,abstracton]{scrartcl}
\usepackage{babel}
\usepackage[utf8]{inputenc}
\usepackage[T1]{fontenc}
\usepackage[a4paper,left=20mm,right=20mm,top=30mm,bottom=30mm]{geometry}
\usepackage{amsmath}
\usepackage{amsthm}
\usepackage{amssymb}
\usepackage{enumerate}
\usepackage[bookmarks=false,
            pdftitle={The Alperin-McKay Conjecture for metacyclic, minimal non-abelian defect groups},
            pdfauthor={Benjamin Sambale},
            pdfkeywords={},
            pdfstartview={FitH}]{hyperref}

\newtheorem{Theorem}{Theorem}[section]
\newtheorem{Lemma}[Theorem]{Lemma}
\newtheorem{Proposition}[Theorem]{Proposition}
\newtheorem{Corollary}[Theorem]{Corollary}

\numberwithin{equation}{section}

\allowdisplaybreaks[1]

\renewcommand{\phi}{\varphi}
\newcommand{\C}{\operatorname{C}}
\newcommand{\N}{\operatorname{N}}
\newcommand{\Z}{\operatorname{Z}}

\newcommand{\Aut}{\operatorname{Aut}}
\newcommand{\Inn}{\operatorname{Inn}}
\newcommand{\Out}{\operatorname{Out}}
\newcommand{\pcore}{\operatorname{O}}

\newcommand{\Irr}{\operatorname{Irr}}

\newcommand{\Bl}{\operatorname{Bl}}

\newcommand{\Gal}{\operatorname{Gal}}
\newcommand{\Syl}{\operatorname{Syl}}

\mathchardef\ordinarycolon\mathcode`\:  
 \mathcode`\:=\string"8000
 \begingroup \catcode`\:=\active
   \gdef:{\mathrel{\mathop\ordinarycolon}}
 \endgroup

\title{The Alperin-McKay Conjecture for metacyclic, minimal non-abelian defect groups}
\author{Benjamin Sambale}
\date{\today}

\begin{document}
\frenchspacing
\maketitle
\begin{abstract}\noindent
We prove the Alperin-McKay Conjecture for all $p$-blocks of finite groups with metacyclic, minimal non-abelian defect groups. These are precisely the metacyclic groups whose derived subgroup have order $p$. 
In the special case $p=3$, we also verify Alperin's Weight Conjecture for these defect groups. Moreover, in case $p=5$ we do the same for the non-abelian defect groups $C_{25}\rtimes C_{5^n}$. The proofs do \emph{not} rely on the classification of the finite simple groups.
\end{abstract}


\section{Introduction}
Let $B$ be a $p$-block of a finite group $G$ with respect to an algebraically closed field of characteristic $p$. Suppose that $B$ has a metacyclic defect group $D$. 
We are interested in the number $k(B)$ (respectively $k_i(B)$) of irreducible characters of $B$ (of height $i\ge 0$), and the number $l(B)$ of irreducible Brauer characters of $B$.
If $p=2$, these invariants are well understood and the major conjectures are known to be true by work of several authors (see \cite{Brauer,Olsson,Robinsonmetac,Sambale,Erdmann,EKKS}). Thus we will focus on the case $p>2$ in the present work. 
Here at least Brauer's $k(B)$-Conjecture, Olsson's Conjecture and Brauer's Height Zero Conjecture are satisfied for $B$ (see \cite{GaoBrauer,YangOlsson,SambaleHZC}). By a result of Stancu~\cite{Stancu}, $B$ is a controlled block.
Moreover, if $D$ is a non-split extension of two cyclic groups, it is known that $B$ is nilpotent (see \cite{Dietz}). Then a result by Puig~\cite{Puig} describes the source algebra of $B$ in full detail.
Thus we may assume in the following that $D$ is a split extension of two cyclic groups. A famous theorem by Dade~\cite{Dade} handles the case where $D$ itself is cyclic by making use of Brauer trees.
The general situation is much harder -- even the case $D\cong C_3\times C_3$ is still open (see \cite{Kiyota,WatanabeSD16,KoshitaniMiyachi,KK}). 
Now consider the subcase where $D$ is non-abelian. Then a work by An~\cite{AnControlled} shows that $G$ is not a quasisimple group. On the other hand, the algebra structure of $B$ in the $p$-solvable case can be obtained from Külshammer~\cite{Kpsolv}.
If $B$ has maximal defect (i.\,e. $D\in\Syl_p(G)$), the block invariants of $B$ were determined in \cite{Gaofull}. If $B$ is the principal block, Horimoto and Watanabe \cite{WatanabePerfIso} constructed a perfect isometry between $B$ and its Brauer correspondent in $\N_G(D)$. 

Let us suppose further that $D$ is a split extension of a cyclic group and a group of order $p$ (i.\,e. $D$ is the unique non-abelian group with a cyclic subgroup of index $p$). Here the difference $k(B)-l(B)$ is known from \cite{GaoZeng}. Moreover, under additional assumptions on $G$, Holloway, Koshitani and Kunugi~\cite{Holloway} obtained the block invariants precisely. In the special case where $D$ has order $p^3$, incomplete information are given by Hendren~\cite{Hendren2}. Finally, one has full information in case $|D|=27$ by work of the present author~\cite{SambaleHZC}.

In the present work we consider the following class of non-abelian split metacyclic groups 
\begin{equation}\label{presmet}
D=\langle x,y\mid x^{p^m}=y^{p^n}=1,\ yxy^{-1}=x^{1+p^{m-1}}\rangle\cong C_{p^m}\rtimes C_{p^n}
\end{equation}
where $m\ge 2$ and $n\ge 1$. By a result of Rédei (see \cite[Aufgabe~III.7.22]{Huppert}) these are precisely the metacyclic, minimal non-abelian groups. A result by Knoche (see \cite[Aufgabe~III.7.24]{Huppert}) implies further that these are exactly the metacyclic groups with derived subgroup of order $p$.
In particular the family includes the non-abelian group with a cyclic subgroup of index $p$ mentioned above. 
The main theorem of the present paper states that $k_0(B)$ is locally determined. In particular the Alperin-McKay Conjecture holds for $B$. This improves some of the results mentioned above. We also prove that every irreducible character of $B$ has height $0$ or $1$. This is in accordance with the situation in $\Irr(D)$.
In the second part of the paper we investigate the special case $p=3$.
Here we are able to determine $k(B)$, $k_i(B)$ and $l(B)$. This gives an example of Alperin's Weight Conjecture and the Ordinary Weight Conjecture. Finally, we determine the block invariants for $p=5$ and $D\cong C_{25}\rtimes C_{5^n}$ where $n\ge 1$.

As a new ingredient (compared to \cite{SambaleHZC}) we make use of the focal subgroup of $B$.





\section{The Alperin-McKay Conjecture}

Let $p$ be an odd prime, and let $B$ be a $p$-block with split metacyclic, non-abelian defect group $D$. Then $D$ has a presentation of the form
\[D=\langle x,y\mid x^{p^m}=y^{p^n}=1,\ yxy^{-1}=x^{1+p^l}\rangle\]
where $0<l<m$ and $m-l\le n$. 
Elementary properties of $D$ are stated in the following lemma.

\begin{Lemma}\hfill
\begin{enumerate}[(i)]
\item $D'=\langle x^{p^l}\rangle\cong C_{p^{m-l}}$.
\item $\Z(D)=\langle x^{p^{m-l}}\rangle\times\langle y^{p^{m-l}}\rangle\cong C_{p^l}\times C_{p^{n-m+l}}$.
\end{enumerate}
\end{Lemma}
\begin{proof}
Omitted.
\end{proof}

We fix a Sylow subpair $(D,b_D)$ of $B$. Then the conjugation of subpairs $(Q,b_Q)\le(D,b_D)$ forms a saturated fusion system $\mathcal{F}$ on $D$ (see \cite{AKO}). Here $Q\le D$ and $b_Q$ is a uniquely determined block of $\C_G(Q)$. 
We also have subsections $(u,b_u)$ where $u\in D$ and $b_u:=b_{\langle u\rangle}$.
By \cite{Stancu}, $\mathcal{F}$ is controlled. Moreover by Theorem~2.5 in \cite{GaoBrauer} we may assume that the inertial group of $B$ has the form $\N_G(D,b_D)/\C_G(D)=\Aut_{\mathcal{F}}(D)=\langle \Inn(D),\alpha\rangle$ where $\alpha\in\Aut(D)$ such that $\alpha(x)\in\langle x\rangle$ and $\alpha(y)=y$. By a slight abuse of notation we often write $\Out_{\mathcal{F}}(D)=\langle\alpha\rangle$.
In particular the inertial index $e(B):=\lvert\Out_{\mathcal{F}}(D)\rvert$ is a divisor of $p-1$. 
Let
\[\mathfrak{foc}(B):=\langle f(a)a^{-1}: a\in Q\le D,\ f\in\Aut_{\mathcal{F}}(Q)\rangle\]
be the \emph{focal subgroup} of $B$ (or of $\mathcal{F}$). Then it is easy to see that $\mathfrak{foc}(B)\subseteq \langle x\rangle$. In case $e(B)=1$, $B$ is nilpotent and $\mathfrak{foc}(B)=D'$. Otherwise $\mathfrak{foc}(B)=\langle x\rangle$.

For the convenience of the reader we collect some estimates on the block invariants of $B$. 

\begin{Proposition}\label{inequ}
Let $B$ be as above. Then
\begin{align*}
\biggl(\frac{p^l+p^{l-1}-p^{2l-m-1}-1}{e(B)}+e(B)\biggr)p^n&\le k(B)\le\biggl(\frac{p^l-1}{e(B)}+e(B)\biggr)(p^{n+m-l-2}+p^n-p^{n-2}),\\
2p^n&\le k_0(B)\le\biggl(\frac{p^l-1}{e(B)}+e(B)\biggr)p^n,\\
\sum_{i=0}^{\infty}{p^{2i}k_i(B)}&\le\biggl(\frac{p^l-1}{e(B)}+e(B)\biggr)p^{n+m-l},\\
l(B)&\ge e(B)\mid p-1,\\
p^n\mid k_0(B),&\hspace{5mm}p^{n-m+l}\mid k_i(B) \hspace{5mm}\text{for }i\ge 1,\\
k_i(B)&=0 \hspace{5mm}\text{for }i>2(m-l).
\end{align*}
\end{Proposition}
\begin{proof}
Most of the inequalities are contained in Proposition~2.1 to Corollary~2.5 in \cite{SambaleHZC}. By Theorem~1 in \cite{RobinsonFocal} we have $p^n\mid|D:\mathfrak{foc}(B)|\mid k_0(B)$. In particular $p^n\le k_0(B)$. In case $k_0(B)=p^n$ it follows from \cite{KLN} that $B$ is nilpotent. However then we would have $k_0(B)=|D:D'|=p^{n+l}>p^n$. Therefore $2p^n\le k_0(B)$. Theorem~2 in \cite{RobinsonFocal} implies $p^{n-m+l}\mid\lvert\Z(D):\Z(D)\cap\mathfrak{foc}(B)\rvert\mid k_i(B)$ for $i\ge 1$.
\end{proof}

Now we consider the special case where $m=l+1$. As mentioned in the introduction these are precisely the metacyclic, minimal non-abelian groups. We prove the main theorem of this section.

\begin{Theorem}\label{AMC}
Let $B$ be a $p$-block of a finite group with metacyclic, minimal non-abelian defect groups for an odd prime $p$. Then 
\[k_0(B)=\biggl(\frac{p^{m-1}-1}{e(B)}+e(B)\biggr)p^n\]
with the notation from \eqref{presmet}. In particular the Alperin-McKay Conjecture holds for $B$.
\end{Theorem}
\begin{proof}
By Proposition~\ref{inequ} we have 
\[p^n\mid k_0(B)\le \biggl(\frac{p^{m-1}-1}{e(B)}+e(B)\biggr)p^n.\]
Thus, by way of contradiction we may assume that
\[k_0(B)\le\biggl(\frac{p^{m-1}-1}{e(B)}+e(B)-1\biggr)p^n.\]
We also have
\[k(B)\ge \biggl(\frac{p^{m-1}+p^{m-2}-p^{m-3}-1}{e(B)}+e(B)\biggr)p^n\] 
from Proposition~\ref{inequ}. Hence the sum $\sum_{i=0}^{\infty}{p^{2i}k_i(B)}$ will be small if $k_0(B)$ is large and $k_1(B)=k(B)-k_0(B)$. 
This implies the following contradiction
\begin{align*}
\biggl(\frac{p^m-1}{e(B)}+p^2+e(B)-1\biggr)p^n&=\biggl(\frac{p^{m-1}-1}{e(B)}+e(B)-1\biggr)p^n+\biggl(\frac{p^{m-2}-p^{m-3}}{e(B)}+1\biggr)p^{n+2}\\
&\le\sum_{i=0}^{\infty}{p^{2i}k_i(B)}\le\biggl(\frac{p^m-p}{e(B)}+pe(B)\biggr)p^n<\biggl(\frac{p^m-1}{e(B)}+p^2\biggr)p^n.
\end{align*}
Since the Brauer correspondent of $B$ in $\N_G(D)$ has the same fusion system, the Alperin-McKay Conjecture follows.
\end{proof}

Isaacs and Navarro~\cite[Conjecture~D]{IsaacsNavarro} proposed a refinement of the Alperin-McKay Conjecture by invoking Galois automorphisms. We show (as a improvement of Theorem~4.3 in \cite{SambaleHZC}) that this conjecture holds in the special case $|D|=p^3$ of Theorem~\ref{AMC}. We will denote the subset of $\Irr(B)$ of characters of height $0$ by $\Irr_0(B)$.

\begin{Corollary}\label{IN}
Let $B$ be a $p$-block of a finite group $G$ with non-abelian, metacyclic defect group of order $p^3$. Then Conjecture~D in \cite{IsaacsNavarro} holds for $B$.
\end{Corollary}
\begin{proof}
Let $D$ be a defect group of $B$. 
For $k\in\mathbb{N}$, let $\mathbb{Q}_k$ be the cyclotomic field of degree $k$. Let $|G|_{p'}$ be the $p'$-part of the order of $G$. It is well-known that the Galois group $\mathcal{G}:=\Gal(\mathbb{Q}_{|G|}|\mathbb{Q}_{|G|_{p'}})$ acts canonically on $\Irr(B)$.
Let $\gamma\in\mathcal{G}$ be a $p$-element. Then it suffices to show that $\gamma$ acts trivially on $\Irr_0(B)$. 
By Lemma~IV.6.10 in \cite{Feit} it is enough to prove that $\gamma$ acts trivially on the $\mathcal{F}$-conjugacy classes of subsections of $B$ via $^\gamma(u,b_u):=(u^{\overline{\gamma}},b_u)$ where $u\in D$ and $\overline{\gamma}\in\mathbb{Z}$. Since $\gamma$ is a $p$-element, this action is certainly trivial unless $|\langle u\rangle|=p^2$. Here however, the action of $\gamma$ on $\langle u\rangle$ is just the $D$-conjugation. The result follows.
\end{proof}

In the situation of Corollary~\ref{IN} one can say a bit more: By Proposition~3.3 in \cite{SambaleHZC}, $\Irr(B)$ splits into the following families of $p$-conjugate characters:
\begin{itemize}
\item $(p-1)/e(B)+e(B)$ orbits of length $p-1$,
\item two orbits of length $(p-1)/e(B)$,
\item at least $e(B)$ $p$-rational characters.
\end{itemize}
Without loss of generality, let $e(B)>1$. By Theorem~4.1 in \cite{SambaleHZC} we have $k_1(B)\le (p-1)/e(B)+e(B)-1$. Moreover, Proposition~4.1 of the same paper implies $k_1(B)<p-1$.
In particular, all orbits of length $p-1$ of $p$-conjugate characters must lie in $\Irr_0(B)$. In case $e(B)=p-1$ the remaining $(p-1)/e(B)+e(B)$ characters in $\Irr_0(B)$ must be $p$-rational. Now let $e(B)<\sqrt{p-1}$. Then it is easy to see that $\Irr_0(B)$ contains just one orbit of length $(p-1)/e(B)$ of $p$-conjugate characters. Unfortunately, it is not clear if this also holds for $e(B)\ge\sqrt{p-1}$.

Next we improve the bound coming from Proposition~\ref{inequ} on the heights of characters.

\begin{Proposition}\label{k2}
Let $B$ be a $p$-block of a finite group with metacyclic, minimal non-abelian defect groups. Then $k_1(B)=k(B)-k_0(B)$. In particular, $B$ satisfies the following conjectures:
\begin{itemize}
\item Eaton's Conjecture \cite{Eaton}
\item Eaton-Moretó Conjecture \cite{EatonMoreto}
\item Robinson's Conjecture \cite[Conjecture 4.14.7]{LuxPahlings} 
\item Malle-Navarro Conjecture \cite{MalleNavarro}
\end{itemize}
\end{Proposition}
\begin{proof}
By \cite{Sambale} we may assume $p>2$ as before.
By way of contradiction suppose that $k_i(B)>0$ for some $i\ge 2$. Since
\[k(B)\ge\biggl(\frac{p^{m-1}+p^{m-2}-p^{m-3}-1}{e(B)}+e(B)\biggr)p^n,\]
we have $k(B)-k_0(B)\ge (p^{m-1}-p^{m-2})p^{n-1}/e(B)$ by Theorem~\ref{AMC}. By Proposition~\ref{inequ}, $k_1(B)$ and $k_i(B)$ are divisible by $p^{n-1}$. This shows
\[\biggl(\frac{p^{m-1}-1}{e(B)}+e(B)\biggr)p^n+\biggl(\frac{p^{m-1}-p^{m-2}}{e(B)}-1\biggr)p^{n+1}+p^{n+3}\le\sum_{i=0}^{\infty}{p^{2i}k_i(B)}\le \biggl(\frac{p^{m-1}-1}{e(B)}+e(B)\biggr)p^{n+1}.\]
Hence, we derive the following contradiction
\[p^{n+3}-p^{n+1}\le\biggl(\frac{1-p}{e(B)}+e(B)(p-1)\biggr)p^n\le p^{n+2}.\]
This shows $k_1(B)=k(B)-k_0(B)$. Now Eaton's Conjecture is equivalent to Brauer's $k(B)$-Conjecture and Olsson's Conjecture. Both are known to hold for all metacyclic defect groups. Also the Eaton-Moretó Conjecture and Robinson's Conjecture is trivially satisfied for $B$. The Malle-Navarro Conjecture asserts that $k(B)/k_0(B)\le k(D')=p$ and $k(B)/l(B)\le k(D)$. By Theorem~\ref{AMC} and Proposition~\ref{inequ}, the first inequality reduces to $p^{n-1}+p^n-p^{n-2}\le p^{n+1}$ which is true. For the second inequality we observe that every conjugacy class of $D$ has at most $p$ elements, since $|D:\Z(D)|=p^2$. 
Hence, $k(D)=\lvert\Z(D)\rvert+\frac{|D|-\lvert\Z(D)\rvert}{p}=p^{n+m-1}+p^{n+m-2}-p^{n+m-3}$.
Now Proposition~\ref{inequ} gives
\[\frac{k(B)}{l(B)}\le k(B)\le\biggl(\frac{p^{m-1}-1}{e(B)}+e(B)\biggr)(p^{n-1}+p^n-p^{n-2})\le p^{n+m-1}+p^{n+m-2}-p^{n+m-3}=k(D).\qedhere\]
\end{proof}

We use the opportunity to present a result for $p=3$ and a different class of metacyclic defect groups (where $l=1$ with the notation above).

\begin{Theorem}
Let $B$ be a $3$-block of a finite group $G$ with defect group
\[D=\langle x,y\mid x^{3^m}=y^{3^n}=1,\ yxy^{-1}=x^4\rangle\]
where $2\le m\le n+1$. Then $k_0(B)=3^{n+1}$. In particular, the Alperin-McKay Conjecture holds for $B$.
\end{Theorem}
\begin{proof}
We may assume that $B$ is non-nilpotent.
By Proposition~\ref{inequ} we have $k_0(B)\in\{2\cdot 3^n,3^{n+1}\}$. By way of contradiction, suppose that $k_0(B)=2\cdot 3^n$. Let $P\in\Syl_p(G)$. Since $D/\mathfrak{foc}(B)$ acts freely on $\Irr_0(B)$, there are $3^n$ characters of degree $a|P:D|$, and $3^n$ characters of degree $b|P:D|$ in $B$ for some $a,b\ge 1$ such that $3\nmid a,b$. Hence,
\[\biggl|\sum_{\chi\in\Irr_0(B)}{\chi(1)^2}\biggr|_3=3^n|P:D|^2(a^2+b^2)_3=|P:D|^2|D:\mathfrak{foc}(B)|.\]
Now Theorem~1.1 in \cite{KLN} gives a contradiction.
\end{proof}

A generalization of the argument in the proof shows that in the situation of Proposition~\ref{inequ}, $k_0(B)=2p^n$ can only occur if $p\equiv 1\pmod{4}$.

\section{Lower defect groups}


In the following we use the theory of lower defect groups in order to estimate $l(B)$. We cite a few results from the literature. Let $B$ be a $p$-block of a finite group $G$ with defect group $D$ and Cartan matrix $C$. We denote the multiplicity of an integer $a$ as elementary divisor of $C$ by $m(a)$. Then $m(a)=0$ unless $a$ is a $p$-power. It is well-known that $m(|D|)=1$. Brauer \cite{BrauerLDG} expressed $m(p^n)$ ($n\ge 0$) in terms of \emph{$1$-multiplicities} of lower defect groups (see also Corollary V.10.12 in \cite{Feit}):
\begin{equation}\label{brauerLDG}
m(p^n)=\sum_{R\in\mathcal{R}}{m_B^{(1)}(R)}
\end{equation}
where $\mathcal{R}$ is a set of representatives for the $G$-conjugacy classes of subgroups $R\le D$ of order $p^n$. 
Later \eqref{brauerLDG} was refined by Broué and Olsson by invoking the fusion system $\mathcal{F}$ of $B$.

\begin{Proposition}[Broué-Olsson \cite{BroueOlsson}]\label{LDG}
For $n\ge 0$ we have
\[m(p^n)=\sum_{R\in\mathcal{R}}{m_B^{(1)}(R,b_R)}\]
where $\mathcal{R}$ is a set of representatives for the $\mathcal{F}$-conjugacy classes of subgroups $R\le D$ of order $p^n$. 
\end{Proposition}
\begin{proof}
This is (2S) of \cite{BroueOlsson}.
\end{proof}

In the present paper we do not need the precise (and complicated) definition of the non-negative numbers $m_B^{(1)}(R)$ and $m_B^{(1)}(R,b_R)$. We say that $R$ is a \emph{lower defect group} for $B$ if $m_B^{(1)}(R,b_R)>0$.
In particular, $m_B^{(1)}(D,b_D)=m^{(1)}_B(D)=m(|D|)=1$. 
A crucial property of lower defect groups is that their multiplicities can usually be determined locally. In the next lemma, $b_R^{\N_G(R,b_R)}$ denotes the (unique) Brauer correspondent of $b_R$ in $\N_G(R,b_R)$.

\begin{Lemma}\label{localLDG}
For $R\le D$ and $B_R:=b_R^{\N_G(R,b_R)}$ we have $m_B^{(1)}(R,b_R)=m_{B_R}^{(1)}(R)$. If $R$ is fully $\mathcal{F}$-normalized, then $B_R$ has defect group $\N_D(R)$ and fusion system $\N_{\mathcal{F}}(R)$. 
\end{Lemma}
\begin{proof}
The first claim follows from (2Q) in \cite{BroueOlsson}. For the second claim we refer to Theorem~IV.3.19 in \cite{AKO}.
\end{proof}

Another important reduction is given by the following lemma.

\begin{Lemma}\label{LDGle}
For $R\le D$ we have $\sum_{Q\in\mathcal{R}}{m_{B_R}^{(1)}(Q)}\le l(b_R)$ where $\mathcal{R}$ is a set of representatives for the $\N_G(R,b_R)$-conjugacy classes of subgroups $Q$ such that $R\le Q\le \N_D(R)$.
\end{Lemma}
\begin{proof}
This is implied by Theorem 5.11 in \cite{OlssonLDG} and the remark following it. Notice that in Theorem~5.11 it should read $B\in\Bl(G)$ instead of $B\in\Bl(Q)$.
\end{proof}

In the local situation for $B_R$ also the next lemma is useful.

\begin{Lemma}\label{centralLDG}
If $\pcore_p(\Z(G))\nsubseteq R$, then $m_B^{(1)}(R)=0$.
\end{Lemma}
\begin{proof}
See Corollary~3.7 in \cite{OlssonLDG}.
\end{proof}

Now we apply these results.

\begin{Lemma}\label{lowermeta}
Let $B$ be a $p$-block of a finite group with metacyclic, minimal non-abelian defect group $D$ for an odd prime $p$.
Then every lower defect group of $B$ is $D$-conjugate either to $\langle y\rangle$, $\langle y^p\rangle$, or to $D$.
\end{Lemma}
\begin{proof}
Let $R<D$ be a lower defect group of $B$. Then $m(|R|)>0$ by Proposition~\ref{LDG}.
Corollary~5 in \cite{RobinsonFocal} shows that $p^{n-1}\mid|R|$. 
Since $\mathcal{F}$ is controlled, the subgroup $R$ is fully $\mathcal{F}$-centralized and fully $\mathcal{F}$-normalized. 
The fusion system of $b_R$ (on $\C_D(R)$) is given by $\C_{\mathcal{F}}(R)$ (see Theorem~IV.3.19 in \cite{AKO}). Suppose for the moment that $\C_{\mathcal{F}}(R)$ is trivial. Then $b_R$ is nilpotent and $l(b_R)=1$. 
Let $B_R:=b_R^{\N_G(R,b_R)}$. Then $B_R$ has defect group $\N_D(R)$ and $m^{(1)}_{B_R}(\N_D(R))=1$. Hence, Lemmas~\ref{localLDG} and \ref{LDGle} imply $m^{(1)}_B(R,b_R)=m^{(1)}_{B_R}(R)=0$. This contradiction shows that $\C_{\mathcal{F}}(R)$ is non-trivial. In particular $R$ is centralized by a non-trivial $p'$-automorphism $\beta\in\Aut_{\mathcal{F}}(D)$. By the Schur-Zassenhaus Theorem, $\beta$ is $\Inn(D)$-conjugate to a power of $\alpha$. Thus, $R$ is $D$-conjugate to a subgroup of $\langle y\rangle$. The result follows.
\end{proof}



\begin{Proposition}\label{lBoben}
Let $B$ be a $p$-block of a finite group with metacyclic, minimal non-abelian defect groups for an odd prime $p$. Then $e(B)\le l(B)\le 2e(B)-1$. 
\end{Proposition}
\begin{proof}
Let 
\[D=\langle x,y\mid x^{p^m}=y^{p^n}=1,\ yxy^{-1}=x^{1+p^{m-1}}\rangle\]
be a defect group of $B$. We argue by induction on $n$. Let $n=1$. By Proposition~\ref{inequ} we have $e(B)\le l(B)$ and
\[k(B)\le \biggl(\frac{p^{m-1}-1}{e(B)}+e(B)\biggr)(1+p-p^{-1}).\]
Moreover, Theorem~3.2 in \cite{SambaleHZC} gives
\[k(B)-l(B)=\frac{p^m+p^{m-1}-p^{m-2}-p}{e(B)}+e(B)(p-1)\]
Hence,
\begin{align*}
l(B)&=k(B)-(k(B)-l(B))\le \biggl(\frac{p^{m-1}-1}{e(B)}+e(B)\biggr)(1+p-p^{-1})-\frac{p^m+p^{m-1}-p^{m-2}-p}{e(B)}-e(B)(p-1)\\
&=2e(B)-\frac{1}{p}\biggl(e(B)-\frac{1}{e(B)}\biggr)-\frac{1}{e(B)},
\end{align*}
and the claim follows in this case.

Now suppose $n\ge 2$.
We determine the multiplicities of the lower defect groups by using Lemma~\ref{lowermeta}. As usual $m(|D|)=1$. Consider the subpair $(\langle y\rangle,b_y)$. By Lemmas~\ref{LDG} and \ref{localLDG} we have $m(p^n)=m^{(1)}_B(\langle y\rangle,b_y)=m^{(1)}_{B_y}(\langle y\rangle)$ where $B_y:=b_y^{\N_G(\langle y\rangle,b_y)}$. Since $\N_D(\langle y\rangle)=\C_D(y)$, it follows easily that $\N_G(\langle y\rangle,b_y)=\C_G(y)$ and $B_y=b_y$. 
By Theorem~IV.3.19 in \cite{AKO} the block $b_y$ has defect group $\C_D(y)$ and fusion system $\C_{\mathcal{F}}(\langle y\rangle)$. In particular $e(b_y)=e(B)$. 
It is well-known that $b_y$ dominates a block $\overline{b_y}$ of $\C_G(y)/\langle y\rangle$ with cyclic defect group $\C_D(y)/\langle y\rangle$ and $e(\overline{b_y})=e(b_y)=e(B)$ (see \cite[Theorem~5.8.11]{Nagao}). By Dade's Theorem on blocks with cyclic defect groups we obtain $l(b_y)=e(B)$. 
Moreover, the Cartan matrix of $\overline{b_y}$ has elementary divisors $1$ and $\lvert\C_D(y)/\langle y\rangle\rvert$ where $1$ occurs with multiplicity $e(B)-1$ (this follows for example from \cite{DetCartan}). Therefore, the Cartan matrix of $b_y$ has elementary divisors $p^n$ and $\lvert\C_D(y)\rvert$ where $p^n$ occurs with multiplicity $e(B)-1$. 
Since $\langle y\rangle\subseteq\Z(\C_G(y))$, Lemma~\ref{centralLDG} implies $m(p^n)=m^{(1)}_{b_y}(\langle y\rangle)=e(B)-1$.

Now consider $(\langle u\rangle,b_u)$ where $u:=y^p\in\Z(D)$. Here $b_u$ has defect group $D$. By the first part of the proof (the case $n=1$) we obtain $l(b_u)=l(\overline{b_u})\le 2e(B)-1$. As above we have $m(p^{n-1})=m^{(1)}_B(\langle u\rangle,b_u)=m^{(1)}_{b_u}(\langle u\rangle)$. Since $p^n$ occurs as elementary divisor of the Cartan matrix of $b_u$ with multiplicity $e(B)-1$ (see above), it follows that $m(p^{n-1})=m^{(1)}_{b_u}(\langle u\rangle)\le e(B)-1$. Now $l(B)$ is the sum over the multiplicities of elementary divisors of the Cartan matrix of $B$ which is at most $m(|D|)+m(\langle y\rangle)+m(\langle u\rangle)\le 1+e(B)-1+e(B)-1=2e(B)-1$.
\end{proof}


The next proposition gives a reduction method.

\begin{Proposition}\label{reduction}
Let $p>2$, $m\ge 2$ and $e\mid p-1$ be fixed. Suppose that $l(B)=e$ holds for every block $B$ with defect group
\[D=\langle x,y\mid x^{p^m}=y^p=1,\ yxy^{-1}=x^{1+p^{m-1}}\rangle\]
and $e(B)=e$. Then every block $B$ with $e(B)=e$ and defect group
\[D=\langle x,y\mid x^{p^m}=y^{p^n}=1,\ yxy^{-1}=x^{1+p^{m-1}}\rangle\]
where $n\ge 1$ satisfies the following:
\begin{align*}
k_0(B)&=\biggl(\frac{p^{m-1}-1}{e(B)}+e(B)\biggr)p^n,&k_1(B)&=\frac{p^{m-1}-p^{m-2}}{e(B)}p^{n-1},\\
k(B)&=\biggl(\frac{p^m+p^{m-1}-p^{m-2}-p}{e(B)}+e(B)p\biggr)p^{n-1},&l(B)&=e(B).
\end{align*}
\end{Proposition}
\begin{proof}
We use induction on $n$. In case $n=1$ the result follows from Theorem~3.2 in \cite{SambaleHZC}, Theorem~\ref{AMC} and Proposition~\ref{k2}.

Now let $n\ge 2$. Let $\mathcal{R}$ be a set of representatives for the $\mathcal{F}$-conjugacy classes of elements of $D$. We are going to use Theorem~5.9.4 in \cite{Nagao}. For $1\ne u\in\mathcal{R}$, $b_u$ has metacyclic defect group $\C_D(u)$ and fusion system $\C_{\mathcal{F}}(\langle u\rangle)$. If $\C_{\mathcal{F}}(\langle u\rangle)$ is non-trivial, $\alpha\in\Aut_{\mathcal{F}}(D)$ centralizes a $D$-conjugate of $u$. Hence, we may assume that $u\in\langle y\rangle$ in this case. If $\langle u\rangle=\langle y\rangle$, then $b_u$ dominates a block $\overline{b_u}$ of $\C_G(u)/\langle u\rangle$ with cyclic defect group $\C_D(u)/\langle u\rangle$. Hence, $l(b_u)=l(\overline{b_u})=e(B)$. Now suppose that $\langle u\rangle<\langle y\rangle$. Then by induction we obtain $l(b_u)=l(\overline{b_u})=e(B)$. Finally assume that $\C_{\mathcal{F}}(\langle u\rangle)$ is trivial. Then $b_u$ is nilpotent and $l(b_u)=1$. It remains to determine $\mathcal{R}$. The powers of $y$ are pairwise non-conjugate in $\mathcal{F}$. As in the proof of Proposition~\ref{k2}, $D$ has precisely $p^{n+m-3}(p^2+p-1)$ conjugacy classes. 
Let $C$ be one of these classes which do not intersect $\langle y\rangle$. 
Assume $\alpha^i(C)=C$ for some $i\in\mathbb{Z}$ such that $\alpha^i\ne 1$. Then there are elements $u\in C$ and $w\in D$ such that $\alpha^i(u)=wuw^{-1}$. Hence $\gamma:=w^{-1}\alpha^i\in\N_G(D,b_D)\cap\C_G(u)$. Since $\gamma$ is not a $p$-element, we 
conclude that $u$ is conjugate to a power of $y$ which was excluded. This shows that no nontrivial power of $\alpha$ can fix $C$ as a set. Thus, all these conjugacy classes split in 
\[\frac{p^2+p-p^{3-m}-1}{e(B)}p^{n+m-3}\]
orbits of length $e(B)$ under the action of $\Out_{\mathcal{F}}(D)$.
Now Theorem~5.9.4 in \cite{Nagao} implies
\[k(B)-l(B)=\biggl(\frac{p^{m-1}+p^{m-2}-p^{m-3}-1}{e(B)}+e(B)\biggr)p^n-e(B).\]
By Proposition~\ref{lBoben} it follows that 
\begin{equation}\label{k+}
k(B)\le \biggl(\frac{p^m+p^{m-1}-p^{m-2}-p}{e(B)}+e(B)p\biggr)p^{n-1}+e(B)-1.
\end{equation}
By Proposition~\ref{inequ} the left hand side of \eqref{k+} is divisible by $p^{n-1}$. Since $e(B)-1<p^{n-1}$, we obtain the exact value of $k(B)$. It follows that $l(B)=e(B)$. Finally, Theorem~\ref{AMC} and Proposition~\ref{k2} give $k_i(B)$.
\end{proof}

For $p=3$, Proposition~\ref{lBoben} implies $l(B)\le 3$. Here we are able to determine all block invariants. 

\begin{Theorem}\label{p3}
Let $B$ be a non-nilpotent $3$-block of a finite group with metacyclic, minimal non-abelian defect groups. Then
\begin{align*}
k_0(B)&=\frac{3^{m-2}+1}{2}3^{n+1},&k_1(B)&=3^{m+n-3},\\
k(B)&=\frac{11\cdot3^{m-2}+9}{2}3^{n-1},&l(B)&=e(B)=2
\end{align*}
with the notation from \eqref{presmet}.
\end{Theorem}
\begin{proof}
By Proposition~\ref{reduction} it suffices to settle the case $n=1$. Here the claim holds for $m\le 3$ by Theorem~3.7 in \cite{SambaleHZC}. We will extend the proof of this result in order to handle the remaining $m\ge 4$. 
Since $B$ is non-nilpotent, we have $e(B)=2$. By Theorem~\ref{AMC} we know $k_0(B)=(3^m+9)/2$.
By way of contradiction we may assume that $l(B)=3$ and $k_1(B)=3^{m-2}+1$ (see Theorem~3.4 in \cite{SambaleHZC}). 

We consider the generalized decomposition numbers $d^z_{\chi\phi_z}$ where $z:=x^3\in\Z(D)$ and $\phi_z$ is the unique irreducible Brauer character of $b_z$. Let $d^z:=(d^z_{\chi\phi_z}:\chi\in\Irr(B))$. By the orthogonality relations we have $(d^z,d^z)=3^{m+1}$. As in \cite[Section~4]{HKS} we can write 
\[d^z=\sum_{i=0}^{2\cdot 3^{m-2}-1}{a_i\zeta_{3^{m-1}}^i}\]
for integral vectors $a_i$ and a primitive $3^{m-1}$-th root of unity $\zeta_{3^{m-1}}\in\mathbb{C}$.
Since $z$ is $\mathcal{F}$-conjugate to $z^{-1}$, the vector $d^z$ is real. Hence, the vectors $a_i$ are linearly dependent. More precisely, it turns out that the vectors $a_i$ are spanned by $\{a_j:j\in J\}$ for a subset $J\subseteq\{0,\ldots,2\cdot3^{m-2}-1\}$ such that $0\in J$ and $|J|=3^{m-2}$.

Let $q$ be the quadratic form corresponding to the Dynkin diagram of type $A_{3^{m-2}}$. We set $a(\chi):=(a_j(\chi):j\in J)$ for $\chi\in\Irr(B)$. Since the subsection $(z,b_z)$ gives equality in Theorem~4.10 in \cite{HKS}, we have
\[k_0(B)+9k_1(B)=\sum_{\chi\in\Irr(B)}{q(a(\chi))}\]
for a suitable ordering of $J$. 
This implies $q(a(\chi))=3^{2h(\chi)}$ for $\chi\in\Irr(B)$ where $h(\chi)$ is the height of $\chi$.
Moreover, if $a_0(\chi)\ne 0$, then $a_0(\chi)=\pm3^{h(\chi)}$ by Lemma~3.6 in \cite{SambaleHZC}.
By Lemma~4.7 in \cite{HKS} we have $(a_0,a_0)=27$.

In the next step we determine the number $\beta$ of $3$-rational characters of of height $1$. 
Since $(a_0,a_0)=27$, we have $\beta<4$. 
On the other hand, the Galois group $\mathcal{G}$ of $\mathbb{Q}(\zeta_{\zeta_{3^{m-1}}})\cap\mathbb{R}$ over $\mathbb{Q}$ acts on $d^z$ and the length of every non-trivial orbit is divisible by $3$ (because $\mathcal{G}$ is a $3$-group). This implies $\beta=1$, since $k_1(B)=3^{m-2}+1$. 

In order to derive a contradiction, we repeat the argument with the subsection $(x,b_x)$. Again we get equality in Theorem~4.10, but this time for $k_0(B)$ instead of $k_0(B)+9k_1(B)$. Hence, $d^x(\chi)=0$ for characters $\chi\in\Irr(B)$ of height $1$. Again we can write $d^x=\sum_{i=0}^{2\cdot3^{m-1}-1}{\overline{a_i}\zeta_{3^m}^i}$ where $\overline{a_i}$ are integral vectors. Lemma~4.7 in \cite{HKS} implies $(\overline{a}_0,\overline{a}_0)=9$. 
Using Lemma~3.6 in \cite{SambaleHZC} we also have $\overline{a}_0(\chi)\in\{0,\pm1\}$. By Proposition~3.3 in \cite{SambaleHZC} we have precisely three $3$-rational characters $\chi_1,\chi_2,\chi_3\in\Irr(B)$ of height $0$ (note that altogether we have four $3$-rational characters). Then $a_0(\chi_i)=\pm\overline{a}_0(\chi_i)=\pm1$ for $i=1,2,3$. By \cite[Section~1]{RobinsonFocal} we have $\lambda*\chi_i\in\Irr_0(B)$ and $(\lambda*\chi_i)(u)=\chi_i(u)=d^u_{\chi_i\phi_u}\phi_u(1)$ for $\lambda\in\Irr(D/\mathfrak{foc}(B))\cong C_3$ and $u\in\{x,z\}$. Since this action on $\Irr_0(B)$ is free, we have nine characters $\psi\in\Irr(B)$ such that $a_0(\psi)=\pm\overline{a}_0(\psi)=\pm1$. In particular $(a_0,\overline{a_0})\equiv 1\pmod{2}$. By the orthogonality relations we have $(d^z,d^{x^j})=0$ for all $j\in\mathbb{Z}$ such that $3\nmid j$. Using Galois theory we get the final contradiction $0=(d^z,\overline{a}_0)=(a_0,\overline{a}_0)\equiv 1\pmod{2}$. 
\end{proof}

In the smallest case $D\cong C_9\rtimes C_3$ of Theorem~\ref{p3} even more information on $B$ were given in Theorem~4.5 in \cite{SambaleHZC}.

\begin{Corollary}\label{AWC}
Alperin's Weight Conjecture and the Ordinary Weight Conjecture are satisfied for every $3$-block with metacyclic, minimal non-abelian defect groups.
\end{Corollary}
\begin{proof}
Let $D$ be a defect group of $B$. Since $B$ is controlled, Alperin's Weight Conjecture asserts that $l(B)=l(B_D)$ where $B_D$ is a Brauer correspondent of $B$ in $\N_G(D)$. Since both numbers equal $e(B)$, the conjecture holds.

Now we prove the Ordinary Weight Conjecture in the form of \cite[Conjecture~IV.5.49]{AKO}. Since $\Out_{\mathcal{F}}(D)$ is cyclic, all $2$-cocycles appearing in this version are trivial. Therefore the conjecture asserts that $k_i(B)$ only depends on $\mathcal{F}$ and thus on $e(B)$. Since the conjecture is known to hold for the principal block of the solvable group $G=D\rtimes C_{e(B)}$, the claim follows.
\end{proof}

We remark that Alperin's Weight Conjecture is also true for the abelian defect groups $D\cong C_{3^n}\times C_{3^m}$ where $n\ne m$ (see \cite{Usami23I,UsamiZ2Z2}).


We observe another consequence for arbitrary defect groups.

\begin{Corollary}
Let $B$ be a $3$-block of a finite group with defect group $D$. Suppose that $D/\langle z\rangle$ is metacyclic, minimal non-abelian for some $z\in\Z(D)$. Then Brauer's $k(B)$-Conjecture holds for $B$, i.\,e. $k(B)\le |D|$.
\end{Corollary}
\begin{proof}
Let $(z,b_z)$ be a major subsection of $B$. Then $b_z$ dominates a block $\overline{b_z}$ of $\C_G(z)/\langle z\rangle$ with metacyclic, minimal non-abelian defect group $D/\langle z\rangle$. Hence, Theorem~\ref{p3} implies $l(b_z)=l(\overline{b_z})\le 2$. Now the claim follows from Theorem~2.1 in \cite{Sambalefurther}.
\end{proof}

In the situation of Theorem~\ref{p3} it is straight-forward to distribute $\Irr(B)$ into families of $3$-conjugate and $3$-rational characters (cf. Proposition~3.3 in \cite{SambaleHZC}). However, it is not so easy to see which of these families lie in $\Irr_0(B)$. 

Now we turn to $p=5$.

\begin{Theorem}\label{p5}
Let $B$ be a $5$-block of a finite group with non-abelian defect group $C_{25}\rtimes C_{5^n}$
where $n\ge 1$. Then
\begin{align*}
k_0(B)&=\biggl(\frac{4}{e(B)}+e(B)\biggr)5^n,&k_1(B)&=\frac{4}{e(B)}5^{n-1},\\
k(B)&=\biggl(\frac{24}{e(B)}+5e(B)\biggr)5^{n-1},&l(B)&=e(B).
\end{align*}
\end{Theorem}
\begin{proof}
By Proposition~\ref{reduction} it suffices to settle the case $n=1$. Moreover by Theorem~4.4 in \cite{SambaleHZC} we may assume that $e(B)=4$. Then by Theorem~\ref{AMC} above and Proposition~4.2 in \cite{SambaleHZC} we have $k_0(B)=25$, $1\le k_1(B)\le 3$, $26\le k(B)\le 28$ and $4\le l(B)\le 6$. 
We consider the generalized decomposition numbers $d^z_{\chi\phi_z}$ where $z:=x^5\in\Z(D)$ and $\phi_z$ is the unique irreducible Brauer character of $b_z$. Since all non-trivial powers of $z$ are $\mathcal{F}$-conjugate, the numbers $d^z_{\chi\phi_z}$ are integral. Also, these numbers are non-zero, because $(z,b_z)$ is a major subsection. Moreover, $d^z_{\chi\phi_z}\equiv 0\pmod{p}$ for characters $\chi\in\Irr(B)$ of height $1$ (see Theorem~V.9.4 in \cite{Feit}). Let $d^z:=(d^z_{\chi\phi_z}:\chi\in\Irr(B))$. By the orthogonality relations we have $(d^z,d^z)=125$. 
Suppose by way of contradiction that $k_1(B)>1$. Then it is easy to see that $d^z_{\chi\phi_z}=\pm 5$ for characters $\chi\in\Irr(B)$ of height $1$. 
By \cite[Section~1]{RobinsonFocal}, 
the numbers $d^z_{\chi\phi_z}$ ($\chi\in\Irr_0(B)$) split in five orbits of length $5$ each. Let $\alpha$ (respectively $\beta$, $\gamma$) be the number of orbits of entries $\pm 1$ (respectively $\pm 2$, $\pm 3$) in $d^z$. Then the orthogonality relations reads
\[\alpha+4\beta+9\gamma+5k_1(B)=25.\]
Since $\alpha+\beta+\gamma=5$, we obtain
\[3\beta+8\gamma=20-5k_1(B)\in\{5,10\}.\]
However, this equation cannot hold for any choice of $\alpha,\beta,\gamma$. Therefore we have proved that $k_1(B)=1$. Now Theorem~4.1 in \cite{SambaleHZC} implies $l(B)=4$.
\end{proof}

\begin{Corollary}
Alperin's Weight Conjecture and the Ordinary Weight Conjecture are satisfied for every $5$-block with non-abelian defect group $C_{25}\rtimes C_{5^n}$.
\end{Corollary}
\begin{proof}
See Corollary~\ref{AWC}.
\end{proof}

Unfortunately, the proof of Theorem~\ref{p5} does not work for $p=7$ and $e(B)=6$ (even by invoking the other generalized decomposition numbers). However, we have the following partial result. 


\begin{Proposition}
Let $p\in\{7,11,13,17,23,29\}$ and let $B$ be a $p$-block of a finite group with defect group $C_{p^2}\rtimes C_{p^n}$
where $n\ge 1$. If $e(B)=2$, then
\begin{align*}
k_0(B)&=\frac{p+3}{2}p^n,&k_1(B)&=\frac{p-1}{2}p^{n-1},\\
k(B)&=\frac{p^2+4p-1}{2}p^{n-1},&l(B)&=2.
\end{align*}
\end{Proposition}
\begin{proof}
We follow the proof of Theorem~4.4 in \cite{SambaleHZC} in order to handle the case $n=1$. After that the result follows from Proposition~\ref{reduction}. 

In fact the first part of the proof of Theorem~4.4 in \cite{SambaleHZC} applies to any prime $p\ge 7$. Hence, we know that the generalized decomposition numbers $d^z_{\chi\phi_z}=a_0(\chi)$ for $z:=x^p$ and $\chi\in\Irr_0(B)$ are integral. Moreover,
\[\sum_{\chi\in\Irr_0(B)}{a_0(\chi)^2}=p^2.\]
The action of $D/\mathfrak{foc}(B)$ on $\Irr_0(B)$ shows that the values $a(\chi)$ distribute in $(p+3)/2$ parts of $p$ equal numbers each. Therefore, Eq. (4.1) in \cite{SambaleHZC} becomes
\[\sum_{i=2}^{\infty}{r_i(i^2-1)}=\frac{p-3}{2}\]
for some $r_i\ge 0$. This gives a contradiction.
\end{proof}

\section*{Acknowledgment}
This work is supported by the Carl Zeiss Foundation and the Daimler and Benz Foundation.

\begin{center}
Benjamin Sambale\\
Institut für Mathematik\\
Friedrich-Schiller-Universität\\
07743 Jena\\
Germany\\
\href{mailto:benjamin.sambale@uni-jena.de}{benjamin.sambale@uni-jena.de}
\end{center}

\end{document}